%% file: quicksort_density_tails.tex
\documentclass[11pt,reqno,tbtags,draft]{amsart}
\usepackage{amssymb}
\usepackage{url}
\usepackage[square,numbers]{natbib}
\bibpunct[, ]{[}{]}{;}{n}{,}{,}
\title[QuickSort density tails]
{On the tails of the limiting {\tt QuickSort} density}

\newcommand\urladdrx[1]{{\urladdr{\def~{{\tiny$\sim$}}#1}}}
\author{James Allen Fill}
\address{Department of Applied Mathematics and Statistics,
The Johns Hopkins University,
3400 N.~Charles Street,
Baltimore, MD 21218-2682 USA}
\email{jimfill@jhu.edu}
\urladdrx{http://www.ams.jhu.edu/~fill/}
\thanks{Research of both authors supported by 
the Acheson~J.~Duncan Fund for the Advancement of Research in
Statistics.}

\author{Wei-Chun Hung}
\address{Department of Applied Mathematics and Statistics,
The Johns Hopkins University,
3400 N.~Charles Street,
Baltimore, MD 21218-2682 USA}
\email{whung6@jhu.edu}

\numberwithin{equation}{section}

\allowdisplaybreaks

\theoremstyle{plain}% default
\newtheorem{theorem}{Theorem}[section]
\newtheorem{lemma}[theorem]{Lemma}
\newtheorem{proposition}[theorem]{Proposition}

\theoremstyle{definition}

\newtheorem{remark}[theorem]{Remark}

\theoremstyle{remark}

\newenvironment{romenumerate}[1][-10pt]{% optional argument changes indentation
\addtolength{\leftmargini}{#1}\begin{enumerate}% gives (i), (ii) etc.
 }{\end{enumerate}}

\newcounter{oldenumi}
{\setcounter{oldenumi}{\value{enumi}}
\begin{romenumerate} \setcounter{enumi}{\value{oldenumi}}}
{\end{romenumerate}}

\newcounter{thmenumerate}

\newcounter{xenumerate}   %no left indentation; thus wider lines

\newcommand{\refT}[1]{Theorem~\ref{#1}}

\newcommand{\refL}[1]{Lemma~\ref{#1}}
\newcommand{\refR}[1]{Remark~\ref{#1}}
\newcommand{\refS}[1]{Section~\ref{#1}}

\newcommand\marginal[1]{\marginpar{\raggedright\parindent=0pt\tiny #1}}
\newcommand\REM[1]{{\raggedright\texttt{[#1]}\par\marginal{XXX}}}

\begingroup
  \divide\count255 by 60
  \multiply\count255 by -60
  \advance\count255 by \time
  \ifnum \count255 < 10 \xdef\klockan{\the\count1.0\the\count255}
  \else\xdef\klockan{\the\count1.\the\count255}\fi
\endgroup

\def\rompar(#1){\textup(#1\textup)}    % usage: \rompar(...)

\def\xexp(#1){e^{#1}}

\newcommand\bbR{\mathbb R}

\newcounter{CC}
\newcounter{cc}
\newcommand\Var{\operatorname{Var}}

\newcommand\gam{\gamma}

\newcommand\Fbar{\overline F}
\newcommand\Fu{\underbar{$F$}}

\input quicksort_density_tails_def.tex

\newcommand{\ignore}[1]{}

\newcommand{\ro}[1]{\uppercase\expandafter{\romannumeral #1}}

\begin{document}

\maketitle

\vspace{-.3in}
\begin{center}
July~30, 2018
\end{center}
\vspace{.1in}

\begin{abstract}
We give upper and lower asymptotic bounds for the left tail and for the right tail of the continuous limiting {\tt QuickSort} density~$f$ that are nearly matching in each tail. The bounds strengthen results from a paper of Svante Janson (2015) concerning the corresponding distribution function~$F$.  Furthermore, we obtain similar 
%upper 
bounds on absolute values of derivatives of~$f$ of each order.
\end{abstract}

\section{Introduction}
Let $X_{n}$ denote the (random) number of comparisons when sorting~$n$ distinct numbers using the algorithm {\tt QuickSort}. Clearly $X_0 = 0$, and for $n \geq 1$ we have the recurrence relation
\[
X_{n} \overset{\mathcal{L}}{=} X_{U_{n} - 1} + X^{*}_{n-U_{n}} + n - 1,
\]
where $\overset{\mathcal{L}}{=}$ denotes equality in law (i.e.,\ in distribution); 
$X_{k} \overset{\mathcal{L}}{=} X^{*}_{k}$; the random variable $U_{n}$ is uniformly distributed on 
$\{1,\dots,n\}$; and $U_{n}, X_{0}, \dots , X_{n-1}$, $X^{*}_{0}, \dots , X^{*}_{n-1}$ are all independent. It is well known that 
\[
\mathbb{E}X_{n} = 2\left(n+1\right)H_{n}-4n,
\] where $H_{n}$ is the $n$th harmonic number $H_n := \sum_{k=1}^{n} k^{-1}$ and (from a simple exact expression) that $\Var X_n = (1 + o(1)) (7 - \frac{2 \pi^2}{3}) n^2$. To study distributional asymptotics, we first center and scale $X_{n}$ as follows:
\[
Z_{n} = \frac{X_{n}-\mathbb{E}X_{n}}{n}.
\]
Using the Wasserstein $d_{2}$-metric, R\"osler \cite{rosler1991limit} proved that $Z_{n}$ converges to $Z$ weakly as $n \rightarrow \infty$.  Using a martingale argument, R\'egnier \cite{regnier1989limiting} proved that the slightly renormalized $\frac{n}{n + 1} Z_n$ converges to $Z$ in $L^{p}$ for every finite~$p$, and thus in distribution; equivalently, the same conclusions hold for $Z_n$.  The random variable~$Z$ has everywhere finite moment generating function with $\mathbb{E}Z = 0$ and 
$\Var Z = 7-\left(2\pi ^{2}/3\right)$. Moreover, $Z$ satisfies the distributional identity
\[
Z \overset{\mathcal{L}}{=} U Z + (1-U) Z^* + g(U).
\]
On the right, $Z^* \overset{\mathcal{L}}{=} Z$; $U$ is uniformly distributed on $\left(0,1\right)$; $U, Z, Z^*$ are independent; and 
\[
g(u) := 2u \ln u + 2 (1-u) \ln (1-u) + 1.
\]
Further, the distributional identity together with the condition that $\mathbb{E}Z$ (exists and) vanishes characterizes the limiting {\tt Quicksort} distribution; this was first shown by 
R\"osler~\cite{rosler1991limit} under the additional condition that $\Var Z < \infty$, and later in full by Fill and Janson~\cite{fill2000fixedpoints}.

Fill and Janson \cite{fill2000smoothness} derived basic properties of the limiting {\tt QuickSort} distribution $\mathcal{L}(Z)$.  In particular, they proved that $\mathcal{L}(Z)$ has a (unique) continuous density~$f$ which is everywhere positive and infinitely differentiable, and for every $k \geq 0$ that $f^{(k)}$ is bounded and enjoys superpolynomial decay in both tails, that is, for each 
$p \geq 0$ and $k \geq 0$ there exists a finite constant $C_{p, k}$ such that $\left| f^{(k)}(x) \right| \leq C_{p,k} |x|^{-p}$ for all $x \in \mathbb{R}$.

%We should point out that Knessl and Szpankowski~\cite{knessl1999quicksort} actually gave tail 
%asymptotics for the (continuous) \emph{density}~$f$, not for the distribution function; 
%Janson~\cite{janson2015tails} integrated those asymptotics to get his (1.6)--(1.7).  
% *** I have handwritten notes, dated 2017.03.24, about this.***  
In this paper, we study asymptotics of $f(-x)$ and $f(x)$ as $x \to \infty$.  Janson~\cite{janson2015tails} concerned himself with the corresponding asymptotics for the distribution function~$F$ and wrote this: ``Using non-rigorous methods from applied mathematics (assuming an as yet unverified regularity hypothesis), Knessl and Szpankowski~\cite{knessl1999quicksort} found very precise asymptotics of both the left tail and the right tail.''  Janson specifies these Knessl--Szpankowski asymptotics for~$F$ in his equations (1.6)--(1.7).  But Knessl and Szpankowski actually did more, producing asymptotics for~$f$, which were integrated by Janson to get corresponding asymptotics for~$F$.  We utilize the same abbreviation $\gamma := (2-\frac{1}{\ln 2})^{-1}$ as Janson~\cite{janson2015tails}.  
With the \emph{same} constant $c_3$ as in (1.6) of~\cite{janson2015tails}, the density analogues of (1.6) (omitting the middle expression) and (1.7) of~\cite{janson2015tails} are that, as 
$x \to \infty$, Knessl and Szpankowski~\cite{knessl1999quicksort} find
\begin{equation}\label{1.6}
f(-x) = \exp\left[ - e^{\gam x + c_3 + o(1)} \right]
\end{equation}
for the left tail and
\begin{equation}\label{1.7}
f(x) = \exp[ - x \ln x - x \ln \ln x + (1 + \ln 2) x + o(x)]
\end{equation}
for the right tail.

We will come 
as close to these non-rigorous results for the density as Janson~\cite{janson2015tails} does for the distribution function, 
and we also obtain similar asymptotic bounds for tail suprema of absolute values of derivatives of the density.  Although our asymptotics for~$f$ 
imply the asymptotics for~$F$ in Janson's main Theorem~1.1, it is important to note that 
in the case of upper bounds (but not lower bounds) on~$f$ 
we use his results in the proofs of ours.

The next two theorems are our main results.
%\newpage 

\begin{theorem}
\label{T:main1}
Let $\gamma := (2-\frac{1}{\ln 2})^{-1}$. As $x \rightarrow \infty$, the limiting {\tt QuickSort} density function~$f$ satisfies
\begin{align}
\label{left}
\exp\left[-e^{\gamma x + \ln \ln x +O(1)}\right] 
&\leq f(-x) \leq \exp \left[-e^{\gamma x + O\left(1\right)}\right],\\
\label{right}
\exp [-x\ln x - x\ln \ln x + O(x)] 
&\leq f(x) \leq \exp [- x \ln x + O(x)].
\end{align}
\end{theorem}

To state our second main theorem we let $\Fu (x) := F(-x)$ and $\Fbar(x) := 1 - F(x)$, and for a function $h:\bbR\to\bbR$ we write
\begin{equation}
\label{xnorm}
\|h\|_x := \sup_{t \geq x} |h(t)|.
\end{equation}

\begin{theorem}
\label{T:main2}
Given an integer $k \geq 0$, as $x \rightarrow \infty$ the $k^{\rm th}$ derivative of the limiting 
{\tt QuickSort} distribution function~$F$ satisfies
\begin{align}
\label{kleft}
\exp\left[-e^{\gamma x + \ln \ln x +O(1)}\right] 
&\leq \| \Fu^{\left(k\right)} \|_x \leq \exp \left[-e^{\gamma x + O(1)}\right], \\
\label{kright}
\exp [-x \ln x - (k \vee 1) x \ln \ln x + O(x)] 
&\leq \| \Fbar^{\left(k\right)} \|_x \leq \exp [- x \ln x + O(x)].
\end{align}
\end{theorem}

\begin{remark}
\label{R:nonrigorous}
(a)~Using the monotonicity of~$F$, it is easy to see that the assertions of \refT{T:main2} for $k = 0$ are equivalent to the main Theorem~1.1 of Janson~\cite{janson2015tails}, which agrees with the formulation of our \refT{T:main2} in that case except that the four bounds are on $|\Fu(x)|$ and 
$|\Fbar(x)|$ instead of the tail suprema 
$\|\Fu\|_x$ and $\|\Fbar\|_x$.  Further, our \refT{T:main1} implies the assertions of \refT{T:main2} for 
$k = 1$.   
So we need only prove \refT{T:main1} and \refT{T:main2} for $k \geq 2$. 

(b)~The non-rigorous arguments of Knessl and Szpankowski~\cite{knessl1999quicksort} suggest
that the following asymptotics as $x \to \infty$ obtained by repeated formal differentiation of 
\eqref{1.6}--\eqref{1.7} are correct for every $k \geq 0$:
\begin{align}
\label{leftKS}
f^{(k)}(-x) &= \exp\left[ - e^{\gam x + c_3 + o(1)} \right], \\
\label{rightKS}
f^{(k)}(x) &= (-1)^k \exp[ - x \ln x - x \ln \ln x + (1 + \ln 2) x + o(x)].
\end{align}
But these remain conjectures for now.
Unfortunately, 
for $k \geq 1$
we don't even know how to identify rigorously the asymptotic signs of $f^{(k)}(\mp x)$!  Concerning $k = 1$, it has long been conjectured that~$f$ is unimodal.  This would of course imply that 
$f'(-x) > 0$ and $f'(x) < 0$ for sufficiently large~$x$.
\end{remark}

As already mentioned, Fill and Janson~\cite{fill2000smoothness} proved that or each $p \geq 0$ and 
$k \geq 0$ there exists a finite constant $C_{p, k}$ such that $\left| f^{(k)}(x) \right| \leq C_{p,k} |x|^{-p}$ for all $x \in \mathbb{R}$.  Our technique for proving the upper bounds in Theorems~\ref{T:main1} and~\ref{T:main2} is to use explicit bounds on the constants $C_k := C_{0, k}$ together with the 
Landau--Kolmogorov inequality (see, for example, \cite{stechkin1967inequalities}).
  
Our paper is organized as follows.  In Section~\ref{S:prelims} we deal with preliminaries: We 
%restate (to render this paper self-contained) the asymptotic results of 
%Janson~\cite[Theorem~1.1]{janson2015tails}, 
recall an integral equation for~$f$ that is the starting point for our lower-bound results in \refT{T:main1}, review the Landau--Kolmogorov inequality, and bound $C_k$ explicitly in terms of~$k$.  
Sections~\ref{LL} and~\ref{RL} derive the stated lower bounds on the left and right tails, respectively, of~$f$ using an iterative approach similar to that of Janson \cite{janson2015tails} for the distribution function. 
In Section \ref{HL} we establish the left-tail results claimed in~\eqref{left} and~\eqref{kleft}.  
In Section \ref{HR}, we establish the right-tail results claimed in~\eqref{right} and~\eqref{kright}.
\newpage 
\section{Preliminaries}
\label{S:prelims}

\ignore{
\subsection{Janson's asymptotic bounds on~$F$}
\label{S:jansonF}
The upper bounds in the following main Theorem 1.1 of Janson~\cite{janson2015tails} are used in our proof of the upper bounds in our Theorems~\ref{T:main1} and~\ref{T:main2} AND IN PROOFS OF LOWER BOUNDS ON HIGHER DERIVATIVES IN BOTH TAILS.

\begin{proposition}
\label{P:Jansonprop}
Let $\gamma := (2-\frac{1}{\ln 2})^{-1}$. As $x \rightarrow \infty$, the limiting {\tt QuickSort} distribution function~$F$ satisfies
\begin{align}
\label{leftF}
\exp\left[-e^{\gamma x + \ln \ln x +O(1)}\right] 
&\leq F(-x) \leq \exp \left[-e^{\gamma x + O\left(1\right)}\right],\\
\label{rightF}
\exp [-x\ln x - x\ln \ln x + O(x)] 
&\leq 1 - F(x) \leq \exp [- x \ln x + O(x)].
\end{align}
\end{proposition}
}

\subsection{An integral equation for~$f$}
\label{S:inteq}
Fill and Janson \cite[Theorem 4.1 and~(4.2)]{fill2000smoothness} produced an integral equation satisfied by~$f$, namely,
\begin{equation}
\label{f integral equation}
f(x) 
= \int_{u=0}^1 \int_{z\in \mathbb{R}} f(z)\,f\!\left(\frac{x-g(u)-(1-u)z}{u}\right)\frac{1}{u}\,dz\,du.
\end{equation}
This integral equation will be used in the proofs of our lower-bound results for~$f$.

\subsection{Landau--Kolmogorov inequality}
\label{S:LK}
For an overview of the Landau--Kolmogorov inequality, see~\cite[Chapter~1]{MPFbook}.  Here we state a version of the inequality well-suited to our purposes; see~\cite{Matorin} and~\cite[display~(21) and the display following~(17)]{stechkin1967inequalities}.

\begin{lemma}
\label{L:LK}
Let $n \geq 2$, and suppose $h:\bbR \to \bbR$ has~$n$ derivatives.  If~$h$ and $h^{(n)}$ are both bounded, then for $1 \leq k < n$ so is $h^{(k)}$.  Moreover, there exist constants $c_{n, k}$ (not depending on~$h$) such that, for every $x \in \bbR$, the supremum norm $\|\cdot\|_x$ defined at~\eqref{xnorm} satisfies
\[
\|h^{(k)}\|_x \leq c_{n, k}\,\|h\|_x^{1 - (k / n)}\,\|h^{(n)}\|_x^{k / n}, \quad 1 \leq k < n.
\]
Further, for $1 \leq k \leq n / 2$ the best constants $c_{n, k}$ satisfy
\[
c_{n, k} \leq n^{(1 / 2) [1 - (k / n)]} (n - k)^{- 1 / 2} \left(\frac{e^2 n}{4 k}\right)^k 
\leq \left(\frac{e^2 n}{4 k}\right)^k.
\] 
\end{lemma} 

\subsection{Explicit constant upper bounds for absolute derivatives}
\label{S:poly}
We also make use of the following two results extracted 
from~\cite[Theorem~2.1 and (3.3)]{fill2000smoothness}.

\begin{lemma}
Let~$\phi$ denote the characteristic function 
%for the limiting {\tt QuickSort} distribution.  
corresponding to~$f$.
Then for every real $p \geq 0$ we have
\[
|\phi(t)| \leq 2^{p^2 + 6 p} |t|^{-p} \quad \mbox{\rm for all $t \in \bbR$}.
\] 
\end{lemma}

\begin{lemma}
For every integer $k \geq 0$ we have
\[
\sup_{x \in \bbR} |f^{(k)}(x)| \leq \frac{1}{2 \pi} \int_{t = - \infty}^{\infty} |t|^k\,|\phi(t)|\,dt.
\]
\end{lemma}

Using these two results, it is now easy to bound $f^{(k)}$.

\begin{proposition}
\label{P:fkbound}
For every integer $k \geq 0$ we have
\[
\sup_{x \in \bbR} |f^{(k)}(x)| \leq 2^{k^2 + 10 k + 17}.
\]
\end{proposition}

\begin{proof}
For every integer $k \geq 0$ we have
\begin{align*}
\sup_{x \in \bbR} |f^{(k)}(x)| 
&\leq \frac{1}{2 \pi} \int_{t = - \infty}^{\infty}\!|t|^k\,|\phi(t)|\,dt \\
&\leq \frac{1}{2 \pi} \left[\int_{|t| > 1}\!|t|^k\,|\phi(t)|\,dt + \int_{|t| \leq 1}\!|t|^k\,|\phi(t)|\,dt\right] \\
&\leq \frac{1}{2 \pi} \left[\int_{|t| > 1}\!2^{(k + 2)^2 + 6 (k + 2)} t^{-2}\,dt + \int_{|t| \leq 1}\!|t|^k\,dt\right] \\
&\leq \frac{1}{\pi} \left[2^{k^2 + 10 k + 16} + \frac{1}{k + 1}\right] \leq 2^{k^2 + 10 k + 17},
\end{align*}
as desired.
\end{proof}

\section{Left Tail Lower Bound on~$f$}
\label{LL}

Our iterative approach to finding the left tail lower bound on~$f$ in \refT{T:main1} is similar to the method used by Janson~\cite{janson2015tails} for~$F$.
The following lemma gives us an inequality that is essential in this section; as we shall see, it is established from a recurrence inequality.  For $z \geq 0$ define
\[
m_{z} := \left(\min\limits_{x\in \left[-z, 0\right]} f\left(x\right)\right)\wedge 1. 
\]
\begin{lemma}
\label{recurrence relation of left tail lower bound}
Given $\epsilon \in (0, 1/10)$, let $a \equiv a(\epsilon) := -g\left(\frac{1}{2}-\epsilon\right) > 0$. 
Then for any integer $k \geq 2$ we have
\vspace{-.1in}
\[
m_{ka} \geq \left(2\epsilon^{3} m_{2a}\right)^{2^{k-2}}.
\]
\end{lemma}

We delay the proof of \refL{recurrence relation of left tail lower bound} in order to show next how the lemma leads us to the desired lower bound in \eqref{left} on the left tail of~$f$ by using the same technique as in \cite{janson2015tails} for~$F$.

\begin{proposition}
As $x \to \infty$ we have
\[
\ln f(-x) \geq -e^{\gamma x + \ln \ln x +O(1)}.
\]
\end{proposition}

\begin{proof}
By Lemma \ref{recurrence relation of left tail lower bound}, for $x > a$ we have
\[
f(-x) \geq m_x \geq m\left( \left\lceil \frac{x}{a} \right\rceil a \right) 
\geq \left(2\epsilon^{3}m_{2a}\right)^{2^{\lceil x/a \rceil-2}} \geq \left(2\epsilon^{3}m_{2a}\right)^{2^{x/a}},
\]
provided $\epsilon$ is sufficiently small that $2 \epsilon^3 m_{2 a} < 1$.
The same as Janson \cite{janson2015tails}, we pick $\epsilon = x^{-1/2}$ and, setting 
$\gamma = (2-\frac{1}{\ln 2})^{-1}$, get $\frac{1}{a} = \frac{\gamma}{\ln 2}+O(x^{-1})$ and
\begin{align*}
\ln f(-x)
&\geq 2^{\frac{\gamma}{\ln 2}x+O(1)} \cdot \ln \left(2\epsilon^{3}m_{2a}\right) \\ 
&= e^{\gamma x + O(1)}\cdot \left( -\mbox{$\frac{3}{2}$} \ln x + \ln m_{2a} + \ln 2 \right)\\
&\geq -e^{\gamma x + \ln \ln x + O(1)}.
\end{align*}
\end{proof}

Now we go back to prove Lemma \ref{recurrence relation of left tail lower bound}:
\begin{proof}[Proof of \refL{recurrence relation of left tail lower bound}]
By the integral equation~\eqref{f integral equation} satisfied by~$f$ (and symmetry in~$u$ about $u = 1/2$), for arbitrary~$z$ and~$a$ we have
\begin{equation}
\label{inteqza}
f(-z-a) 
= 2 \int_{u=0}^{1/2}{\int_{y \in \mathbb{R}} f(y) f\left(\frac{-z-a-g(u)-(1-u)y}{u}\right)}\frac{1}{u}\,dy\,du.
\end{equation}
Since~$f$ is everywhere positive, we can get a lower bound on $f(-z-a)$ by restricting the range of integration in~\eqref{inteqza}. Therefore,
%for $\epsilon < 1$ we have
\begin{equation}
f(-z-a) \geq 2 \int_{u=\frac{1}{2}-\frac{\epsilon}{2}}^{1/2}{\int_{y =-z}^{-z+\epsilon^2}\!f(y) f\!\left(\frac{-z-a-g(u)-(1-u)y}{u}\right)} \frac{1}{u}\,dy\,du.
\label{int_eq for -z-a}
\end{equation}

We claim that in this integral region, we have $\frac{-z-a-g(u)-(1-u)y}{u} \geq -z$, which is equivalent to $y+z \leq \frac{-a-g(u)}{1-u}$. Here is a proof.  Observe that when $\epsilon$ is small enough and 
$u \in [\frac{1}{2}-\frac{\epsilon}{2}, \frac{1}{2}]$, we have
\begin{align*}
\frac{-a-g(u)}{1-u} &\geq \frac{g\left(\frac{1}{2}-\epsilon\right)-g\left(\frac{1}{2} - \frac{\epsilon}{2}\right)}{\frac{1}{2}+\frac{\epsilon}{2}} \\ 
& \geq \frac{\frac{\epsilon}{2}\left|g'\!\left(\frac{1}{2}-\frac{\epsilon}{2}\right)\right|}{\frac{1}{2}+\frac{\epsilon}{2}} = \frac{\epsilon}{1+\epsilon}\left|2\ln \left(1-\frac{2\epsilon}{1+\epsilon}\right)\right|\\
& \geq \frac{4\epsilon^{2}}{\left(1+\epsilon\right)^{2}} \geq \epsilon^{2}.
\end{align*}
Also, in this integral region we have 
%$y \geq -z$ and 
$y+z \leq \epsilon^{2}$. 
So we conclude that $y+z \leq \frac{-a-g(u)}{1-u}$. 

Next, we claim that $\frac{-z-a-g(u)-(1-u)y}{u} \leq 0$ in this integral region if $z$ is large enough.
Here is a proof. 
Let $\frac{-z-a-g(u)-(1-u)y}{u} = -z +\delta$ with $\delta \geq 0$.  Then in the integral region we have $0 \leq y+z = \frac{-a-g(u)-u\delta}{1-u}$. Therefore
\begin{align*}
\delta \leq \frac{-a-g(u)}{u} &\leq \frac{-a-g\left(\frac{1}{2}\right)}{\frac{1}{2}-\frac{\epsilon}{2}} = \frac{2}{1-\epsilon}\left[g\left(\frac{1}{2}-\epsilon\right)-g\left(\frac{1}{2}\right)\right]\\
&\leq \frac{2\epsilon}{1-\epsilon}\left|2\ln \left(1-\frac{4\epsilon}{1+2\epsilon}\right)\right|\\
&\leq 19 \epsilon^{2}, 
\end{align*}
where the last inequality can be verified to hold for $\epsilon < 1/10$. 
That means if we pick $z$ large enough, for example, $z \geq 20 \epsilon^2$, then $\frac{-z-a-g(u)-(1-u)y}{u} = -z + \delta$ will be negative.  It can also be verified that $a \geq 30 \epsilon^2$ for $\epsilon < 1/10$.  

Now consider $\epsilon < 1/10$, an integer $k \geq 3$, $z \in [(k-2)a, (k-1)a]$, and 
$x = z + a \in [(k-1)a, k a]$.  Noting $z \geq a \geq 30 \epsilon^2 > 20 \epsilon^2$,
by~(\ref{int_eq for -z-a}) we have 
\[
f(-x) \geq 2\cdot \frac{\epsilon}{2} \cdot m_z^2 \cdot \epsilon^{2}\cdot 2 \geq 2\epsilon^{3}m_{(k-1)a}^2.
\]
Further, for $x \in [0, (k-1)a]$ 
%and $\epsilon$ small enough so that $2\epsilon^{3} < 1$% 
we have
\[
f\left(-x\right) \geq m_{\left(k-1\right)a} > 2\epsilon^{3}m_{\left(k-1\right)a}^{2}
\]
since $2 \epsilon^3 < 1$ and $m_{(k - 1) a} \leq 1$ by definition.
Combine these two facts, we can conclude that for $x \in \left[0, ka \right]$ we have $f\left(-x\right) \geq 2\epsilon^{3}m_{\left(k-1\right)a}^{2} $. This implies the recurrence inequality
\[
m_{ka} \geq 2\epsilon^{3}m_{\left(k-1\right)a}^{2}.
\]
%We then use the same way to lower bound $m_{\left(k-1\right)a}$ until we deal with $m_{3a}$ and get $m_{3a} \geq 2\epsilon^{3}m_{2a}^{2}$. 
The desired inequality follows by iterating:
\[
m_{ka} \geq \left(2\epsilon^{3} \right)^{2^{k-2}-1}m_{2a}^{2^{k-2}} \geq \left(2\epsilon^{3}\cdot m_{2a}\right)^{2^{k-2}}.
\]
\end{proof}

\section{Right Tail Lower Bound on~$f$}
\label{RL}

Once again we use an iterative approach to derive our right-tail lower bound on~$f$ in \refT{T:main1}.
The following key lemma is established from a recurrence inequality.  
Define
\[
c := 2 [F(1) - F(0)] \in (0, 2)
\]
and
\[
m_{z} := \min \limits_{x\in [0, z]} f(x), \quad z \geq 0. 
\]

\begin{lemma}
\label{recurrence relation of right tail lower bound}
Suppose $b \in [0, 1)$ and that $\delta \in (0, 1/2)$ is sufficiently small that $g(\delta) \geq b$.  Then for any integer $k \geq 1$ satisfying
\[
2 + (k - 1) b \leq [g(\delta) - b] / \delta
\]
we have
\[
m_{2 + k b} \geq (c \delta)^{k - 1} m_3.
\]
\end{lemma}

We delay the proof of \refL{recurrence relation of right tail lower bound} in order to show next how the lemma leads us to the desired lower bound in \eqref{right} on the right tail of~$f$.

\begin{proposition}
As $x \to \infty$ we have
\[
f(x) \geq \exp[ - x \ln x - x \ln \ln x + O(x)].
\]
\end{proposition}

\begin{proof}
Given $x \geq 3$ suitably large, we will show next that we can apply \refL{recurrence relation of right tail lower bound} for suitably chosen $b > 0$ and~$\delta$ and $k = \lceil (x - 2) / b \rceil \geq 2$.  Then, by the lemma, 
\begin{equation}
\label{fxLB}
f(x) \geq m_{2 + k b} \geq (c \delta)^{k - 1} m_3 \geq (c \delta)^{(x - 2) / b} m_3,
\end{equation}
and we will use~\eqref{fxLB} to establish the proposition.

We make the same choices of~$\delta$ and~$b$ as in~\cite[Sec.~4]{janson2015tails}, namely, 
$\delta = 1 / (x \ln x)$ and $b = 1 - (2 / \ln x)$.  To apply \refL{recurrence relation of right tail lower bound}, we need to check that $g(\delta) \geq b$ and $2 + (k - 1) b \leq [g(\delta) - b]/\delta$, for the latter of which it is sufficient that $x \leq [g(\delta) - b]/\delta$.  Indeed, if~$x$ is sufficiently large, then
\[
g(\delta) \geq 1 + 3 \delta \ln \delta = 1 - \mbox{$\frac{3}{x \ln x}$} (\ln x + \ln \ln x) 
\geq 1 - \mbox{$\frac{4}{x}$},
\]
where the elementary first inequality is (4.1) in~\cite{janson2015tails}, and so  
\[
g(\delta) - b \geq \mbox{$\frac{2}{\ln x}$} - \mbox{$\frac{4}{x}$} \geq \mbox{$\frac{1}{\ln x}$} > 0
\]
and
\[
\frac{g(\delta) - b}{\delta} \geq \frac{1 / \ln x}{1 / (x \ln x)} = x.
\]

Finally, we use~\eqref{fxLB} to establish the proposition.  Indeed,
\begin{align*}
- \ln f(x) 
&\leq \mbox{$\frac{x - 2}{b}$} \ln(\mbox{$\frac{1}{c \delta}$}) - \ln m_3 \\
&\leq \mbox{$\frac{x}{1 - (2 / \ln x)}$} [\ln(x \ln x) + \ln(\mbox{$\frac{1}{c}$})] - \ln m_3 \\
&= \mbox{$\frac{x}{1 - (2 / \ln x)}$} \ln(x \ln x) + O(x).
\end{align*}
But
\begin{align*}
\lefteqn{\hspace{-.5in}\frac{x}{1 - (2 / \ln x)} \ln(x \ln x)} \\
&= x \left[ 1 + \frac{2}{\ln x} + O\left(\frac{1}{(\log x)^2}\right) \right] (\ln x + \ln \ln x) \\
&= (x \ln x) \left[ 1 + \frac{2}{\ln x} + O\left(\frac{1}{(\log x)^2}\right) \right] 
  \left( 1 + \frac{\ln \ln x}{\ln x} \right) \\
&= (x \ln x) \left[ 1 +  \frac{\ln \ln x}{\ln x} + \frac{2}{\ln x} + \frac{2 \ln \ln x}{(\ln x)^2} 
  + O\left(\frac{1}{(\log x)^2}\right) \right] \\
&= x \ln x +  x \ln \ln x + 2 x + \frac{2 x \ln \ln x}{\ln x} + O\left(\frac{x}{\log x}\right) \\
&= x \ln x +  x \ln \ln x + O(x).
\end{align*}
So
\[
- \ln f(x) \leq x \ln x +  x \ln \ln x + O(x),
\]
as claimed.
\end{proof}

Now we go back to prove Lemma \ref{recurrence relation of right tail lower bound}, but first we need two preparatory results.

\begin{lemma}
\label{fzb lower bound}
Suppose $z \geq 2$, $b \geq 0$, and $\delta \in (0, 1/2)$ satisfy $g(\delta) \geq b$ and 
$z \leq [g(\delta) - b] / \delta$.  Then
\[
f(z + b) \geq c\,\delta\,m_z.
\]
\end{lemma}

\begin{proof}
By the integral equation~\eqref{f integral equation} satisfied by~$f$ (and symmetry in~$u$ about $u = 1/2$), for arbitrary~$z$ and~$b$ we have
\[
f(z+b) 
= 2 \int_{u=0}^{1/2}{\int_{y \in \mathbb{R}} f(y) f\left(\frac{z+b-g(u)-(1-u)y}{u}\right)}\frac{1}{u}\,dy\,du.
\]
Since $f$ is positive everywhere, a lower bound on $f\left(z+b\right)$ can be achieved by shrinking the region of integration:
\begin{align}
f\left(z+b\right) 
&\geq 2 \int_{u=0}^{\delta}{\int_{y=0}^{z} f(y) f\left(\frac{z+b-g(u)-(1-u)y}{u}\right)}\frac{1}{u}\,dy\,du \nonumber\\
&\geq 2m_{z} \int_{u=0}^{\delta}{\int_{y=0}^{z} f\left(\frac{z+b-g(u)-(1-u)y}{u}\right)}\frac{1}{u}\,dy\,du \nonumber\\
& = 2m_{z} \int_{u=0}^{\delta}{\int_{\xi=z+\frac{b-g(u)}{u}}^{\frac{z + b - g(u)}{u}}
f(\xi)}\frac{1}{1-u}\, d\xi\, du.
\end{align}

The equality comes from a change of variables.  We next claim that the integral of integration 
for~$\xi$ contains $(0, z - 1)$, and then the desired result follows.  Indeed, if $u \in (0, \delta)$ and $\xi \in (0, z - 1)$ then
\[
\xi < z - 1 < \mbox{$\frac{z - 1}{u}$} \leq \mbox{$\frac{z + b - g(u)}{u}$}, 
\]  
where the last inequality holds because $b \geq 0$ and $g(u) \leq 1$; and, because 
$g(u) \geq g(\delta)$ and $g(\delta) \geq b$ and $z \leq [g(\delta) - b] / \delta$, we have
\begin{align*}
\xi 
&> 0 = z + \mbox{$\frac{b - g(u)}{u}$} - [z + \mbox{$\frac{b - g(u)}{u}$}]
\geq z + \mbox{$\frac{b - g(u)}{u}$} - [z + \mbox{$\frac{b - g(\delta)}{u}$}] \\
&\geq z + \mbox{$\frac{b - g(u)}{u}$} - [z + \mbox{$\frac{b - g(\delta)}{\delta}$}]
\geq z + \mbox{$\frac{b - g(u)}{u}$}.
\end{align*}
\end{proof}

\begin{lemma}
\label{m lower bound}
Suppose $b \geq 0$ and that $\delta \in (0, 1/2)$ is sufficiently small that $g(\delta) \geq b$.  Then for any integer $k \geq 2$ satisfying
\[
2 + (k - 1) b \leq [g(\delta) - b] / \delta
\]
we have
\[
m_{2 + k b} \geq c\,\delta\,m_{2 + (k - 1) b}.
\]
\end{lemma}

\begin{proof}
For $y \in [2 + (k - 1)b, 2 + k b]$, application of \refL{fzb lower bound} with $z = y - b$ yields
\[
f(y) \geq c\,\delta\,m_{y - b} \geq c\,\delta\,m_{2 + (k - 1) b}.
\]
Also, for $y \in [0, 2 + (k - 1) b]$ we certainly have 
\[
f(y) \geq m_{2 + (k - 1) b} > c\,\delta\,m_{2 + (k - 1) b}.
\]
The result follows.
\end{proof}

We are now ready to complete this section by proving \refL{recurrence relation of right tail lower bound}. 

\begin{proof}[Proof of \refL{recurrence relation of right tail lower bound}]
By iterating the recurrence inequality of \refL{m lower bound}, it follows that
\[
m_{2 + k b} \geq (c\,\delta)^{k - 1} m_{2 + b}.
\]
\refL{recurrence relation of right tail lower bound} then follows since $b < 1$.
\end{proof}

\section{Left Tail Bounds for Tail Suprema of Absolute Derivatives}
\label{HL}
From \refS{LL} (respectively, \refS{RL}) we know the left-tail lower bound of~\eqref{left} [resp.,\ the right-tail lower bound of~\eqref{right}].
 In this section we establish the left-tail bounds of~\eqref{left} and~\eqref{kleft}, and in the next section we do the same for right tails.

\subsection{Lower bounds}
\label{HLL}
As discussed in \refR{R:nonrigorous}(a), in light of the main theorem of Janson~\cite{janson2015tails} and our \refS{LL}, to finish our treatment of left-tail lower bounds we need only prove the lower bound in~\eqref{kleft} for fixed $k \geq 2$.  For that, choose any~$x$ and apply the Landau--Kolmogorov \refL{L:LK}, bounding the function $\Fu'(\cdot) = -f(-\cdot)$ in terms of the functions $\Fu$ and $\Fu^{(k)}$.  This gives  
\[
f(-x) \leq \|\Fu'\|_x \leq c_{k, 1}\,\|\Fu\|_x^{(k - 1) / k}\,\|\Fu^{(k)}\|_x^{1 / k},
\]
i.e.,
\[
\|\Fu^{(k)}\|_x \geq c_{k, 1}^{-k}\,\|\Fu\|_x^{- (k - 1)}\,[f(-x)]^k.
\]
But recall 
\[
c_{k, 1} \leq e^2 k / 4, \quad \|\Fu\|_x \leq \exp\left[-e^{\gamma x + O(1)}\right], \quad 
f(-x) \geq \exp\left[-e^{\gamma x + \ln \ln x + O(1)}\right].  
\] 
Plugging in these bounds, we obtain the desired result.
  
\subsection{Upper bounds}
\label{HLU}
The left-tail upper bounds in~\eqref{kleft} of~\refT{T:main2} can be written in the equivalent form 
\begin{equation}
\label{kleftequiv}
\lambda_k := \limsup_{x \to \infty} \left[ \gamma x - \ln \left( -\ln \|\Fu^{(k)}\|_x \right) \right] < \infty;
\end{equation}
note also that the left-tail upper bound in~\eqref{left} of~\refT{T:main1} follows from $\lambda_1 < \infty$.  As discussed in \refR{R:nonrigorous}(a), \eqref{kleftequiv} is known for $k = 0$ from Janson~\cite{janson2015tails}.  So to finish our treatment of left-tail upper bounds in Theorems~\ref{T:main1}--\ref{T:main2} we need only prove~\eqref{kleftequiv} for $k \geq 1$.

In this subsection we prove the following stronger Proposition~\ref{P:HLUprop}, which implies that 
$\lambda_k$ is non-increasing in $k \geq 0$ and therefore that $\lambda_k < \infty$ for every~$k$.  In preparation for the proof, see the definition of $\mu_j$ in~\eqref{limsupdiffL} and note that if $\mu_j \leq 0$ for $j = 0, \dots, k - 1$, then $\lambda_j$ is non-increasing for $j = 0, \dots, k$; in particular, \eqref{kleftequiv} then holds.

\begin{proposition}
\label{P:HLUprop}
For each fixed $k \geq 0$ we have
\begin{equation}
\label{limsupdiffL}
\mu_k := \limsup_{x \to \infty} \left[ - \ln \left( -\ln \|\Fu^{(k + 1)}\|_x \right) 
+ \ln \left( -\ln \|\Fu^{(k)}\|_x \right) \right] \leq 0.
\end{equation}
\end{proposition}

\begin{proof}
We proceed by induction on~$k$.
Choosing any~$x$
and applying the Landau--Kolmogorov inequality \refL{L:LK} to the function $h = \Fu^{(k)}$, we find for $n \geq 2$ that
\[
\|\Fu^{(k + 1)}\|_x 
\leq \mbox{$\frac{1}{4}$} e^2 n\,\|\Fu^{(k)}\|_x^{1 - (1 / n)}\,\|\Fu^{(k + n)}\|_x^{1 / n}.
\]
We can bound the norm $\|\Fu^{(k + n)}\|_x$ using Proposition~\ref{P:fkbound} simply by 
\begin{equation}
\label{factor3}
a_{n, k} := 2^{(k + n - 1)^2 + 10 (k + n - 1) + 17}.
\end{equation}
Thus the argument of the $\limsup$ in~\eqref{limsupdiffL} can be bounded above by
\[
- \ln \left[1 - \frac{1}{n} - \frac{2 - \ln 4 + \ln n + n^{-1} \ln a_{n, k}}{- \ln \|\Fu^{(k)}\|_x} \right].
\]
By Janson's bound giving $\lambda_0 < \infty$ if $k = 0$ and by induction on~$k$ if $k \geq 1$, we know that~\eqref{kleftequiv} holds.
Thus, letting $n \equiv n(x) \to \infty$ with $n(x) = o(e^{\gamma x})$, the claimed inequality follows.
\end{proof}

\begin{remark}
According to \refR{R:nonrigorous}, it is natural to conjecture that for every~$k$ the $\limsup$ in~\eqref{kleftequiv} is a limit and equals $- c_3$ and hence the $\limsup$ in~\eqref{limsupdiffL} is a vanishing limit. 
\end{remark}

\section{Right Tail Bounds for Tail Suprema of Absolute Derivatives}
\label{HR}

In this section we establish the right-tail bounds of~\eqref{right} and~\eqref{kright}.

\subsection{Lower bounds}
\label{HRL}

As discussed in \refR{R:nonrigorous}(a), in light of the main theorem of \cite{janson2015tails} and our \refS{RL}, to finish our treatment of right-tail lower bounds we need only prove the lower bound in~\eqref{kright} for fixed $k \geq 2$.  For that, proceed using the Landau--Kolmogorov \refL{L:LK} as in \refS{HLL} to obtain 
\[
\|\Fbar^{(k)}\|_x \geq c_{k, 1}^{-k}\,\|\Fbar\|_x^{- (k - 1)}\,[f(x)]^k.
\]
But recall 
\begin{align*}
c_{k, 1} &\leq e^2 k / 4, \quad 
\|\Fbar\|_x \leq \exp[ - x \ln x + O(x)], \\ 
f(x) &\geq \exp\left[ - x \ln x - x \ln \ln x + O(x) \right].  
\end{align*}
Plugging in these bounds, we obtain the desired result.

\subsection{Upper bounds}
\label{HRU}

The right-tail upper bounds in~\eqref{kright} of~\refT{T:main2} can be written in the equivalent form
\begin{equation}
\label{krightequiv}
\rho_k := \limsup_{x \to \infty}\,x^{-1} \left( x \ln x + \ln \|\Fbar^{(k)}\|_x \right) < \infty;
\end{equation}
note also that the right-tail upper bound in~\eqref{right} of~\refT{T:main1} follows from 
$\rho_1 < \infty$.  As discussed in \refR{R:nonrigorous}(a), \eqref{krightequiv} is known for $k = 0$ from Janson~\cite{janson2015tails}.  So to finish our treatment of right-tail upper bounds in Theorems~\ref{T:main1}--\ref{T:main2} we need only prove~\eqref{krightequiv} for $k \geq 1$.

In this subsection we prove the next stronger Proposition~\ref{P:HRUprop}, a right-tail analogue of Proposition~\ref{P:HLUprop}, and it then follows by choosing $r(x) \equiv x$ that $\rho_k$ is non-increasing in $k \geq 0$ and therefore that $\rho_k < \infty$ for every~$k$.  
%In preparation for the proof, see the definition of 
%$\sigma_j$ in~\eqref{limsupdiffR} and note that if $\sigma_j \leq 0$ for $j = 0, \dots, k - 1$, then 
%$\rho_j$ is non-increasing for $j = 0, \dots, k$; in particular, \eqref{krightequiv} then holds.    

\begin{proposition}
\label{P:HRUprop}
Let~$r$ be a function satisfying $r(x) = \omega(\sqrt{x \log x})$ as $x \to \infty$.  Then for each fixed 
$k \geq 0$ we have
\begin{equation}
\label{limsupdiffR}
\sigma_k := \limsup_{x \to \infty}\,r(x)^{-1} \left( \ln \|\Fbar^{(k + 1)}\|_x - \ln \|\Fbar^{(k)}\|_x \right) \leq 0.
\end{equation}
\end{proposition}

\begin{proof}
Proceeding as in the proof of Proposition~\ref{P:HLUprop}, for any~$x$ and any 
$n \geq 2$ we have
\[
\|\Fbar^{(k + 1)}\|_x 
\leq \mbox{$\frac{1}{4}$} e^2 n\,\|\Fbar^{(k)}\|_x^{1 - (1 / n)}\,\|\Fbar^{(k + n)}\|_x^{1 / n};
\]
we again bound the norm $\|\Fbar^{(k + n)}\|_x$ by~\eqref{factor3}.
Thus the argument of the $\limsup$ in~\eqref{limsupdiffR} can be bounded above by
\[
r(x)^{-1} \left[ \mbox{$\frac{1}{n}$} (- \ln \|\Fbar^{(k)}\|_x ) + 2 - \ln 4 + \ln n + \mbox{$\frac{1}{n}$} 
\ln a_{n, k} \right].
\]
By the right-tail \emph{lower} bound for $\|\Fbar^{(k)}\|_x$ in~\eqref{kright} (established in the preceding subsection), we know that
\[
- \ln \|\Fbar^{(k)}\|_x \leq x \ln x + (k \vee 1) x \ln \ln x + O(x) = (1 + o(1)) x \ln x.
\] 
Thus, letting $n \equiv n(x)$ satisfy $n(x) = \omega((x \log x) / r(x))$ and $n(x) = o(r(x))$, 
the claimed inequality follows.
\end{proof}

\begin{remark}
According to \refR{R:nonrigorous}, it is natural to conjecture that for every~$k$ we have
$\rho_k = - \infty$ and the $\limsup$ in~\eqref{limsupdiffR} with $r(x) \equiv x$ is a vanishing limit. 
\end{remark}

{\bf Acknowledgment.}\ We are grateful to an anonymous referee of an earlier extended abstract version of this paper for alerting us to the Landau--Kolmogorov inequality, which for upper bounds greatly simplified our proofs and improved our results.  We also thank Svante Janson and another anonymous referee for helpful comments.

\bibliography{bib_file}
\bibliographystyle{plain}

\end{document}

%% file: quicksort_density_tails_def.tex
\newcommand\doi{D_{01}}